\documentclass[12pt]{amsart}
\usepackage{amssymb}
\usepackage{amsfonts}
\usepackage{amsmath}
\usepackage[left=3cm,right=2cm,top=2cm,bottom=2cm]{geometry}
\usepackage[unicode]{hyperref}
\newtheorem{thm}{Theorem}[section]
\newtheorem{lem}[thm]{Lemma}
\newtheorem{cor}[thm]{Corrolary}
\DeclareMathOperator{\UT}{UT}
\DeclareMathOperator{\GL}{GL}
\DeclareMathOperator{\T}{T}
\let\le\leqslant
\let\ge\geqslant
\let\ph\varphi
\def\N{\mathbb N}
\def\UP{U\!\!P}
\begin{document}

\title{Bijections preserving commutators and automorphisms of unitriangular group}

\author{Waldemar Ho{\l}ubowski}
\email{w.holubowski@polsl.pl} 
\address{Institute of Mathematics, Silesian University of Technology\\
Kaszubska 23, 44-101 Gliwice, Poland}

\author{Alexei Stepanov}
\email{stepanov239@gmail.com}
\address{Dept. of Mathematics and Mechanics,
St.Petersburg State University,
Universitetsky 28, Peterhof,
198504, St. Petersburg, Russia}
\address
{St.Petersburg State Electrotechnical University,
Prof. Popova 5, 197376, St. Petersburg, Russia}


\thanks
{The work started during the visit of the second author to the Silesian University of Technology
supported by St.Petersburg State University travel grant 6.42.1273.2014 and State Financed
research task 6.38.191.2014. Work of the second author on this publication is partially supported
by the Ministry of  Education and Science of Russia (Contracts 114031340002 and 2.136.2014/K)
and RFBR grant 13-01-00709.}

\begin{abstract}
We complete characterization of bijections preserving commutators (PC-maps) in the group of
unitriangular matrices $\UT(n,F)$ over a field $F$, where $n\in\N\cup\{\infty\}$.
PC-maps were recently described up to almost identity PC-maps by M.\,Chen, D.\,Wang, and H.\,Zhai ($2011$)
for finite $n$ and by R.\,Slowik ($2013$) for $n=\infty$. An almost identity map is a map,
preserving elementary transvections. We show that an almost identity PC-map is a
multiplication by a central element. In particular, if $n=\infty$, then an almost
identity map is identity. Together with the result of R.\,Slowik
this shows that any PC-map of $\UT(\infty, F)$ is an automorphism.
\end{abstract}

\maketitle

\section*{Introduction}

Linear preserver problem is one of the most studied subjects in linear algebra
(see~\cite{LiT} -- \cite{CL}). It aims at characterization of operators on matrix spaces
(or semigroups, groups, Lie algebras) that leave some operations, properties, or relations invariant.
In particular, maps preserving commutativity~\cite{BD, BS1},  nilpotent matrices~\cite{BPW},
normality~\cite{BS2}, regularity~\cite{BKSS}, similarity~\cite{LiP2, RS}, idempotents~\cite{MB3}
were intensively studied.
In this note we focus on property that is stronger than commutativity.
We study bijective maps that preserve the operation of taking commutators, i.\,e.
bijective maps $\ph:G\to G$ such that
\begin{equation}\label{comm}
\ph({x,y}) = [\ph(x), \ph(y)]\text{ for all }x,y\in G,
\end{equation}
where $G$ is a group and $[x,y]= xy x^{-1} y^{-1}$ denotes group commutator.
Bijections satisfying (\ref{comm}) are called PC-maps. The
set $\operatorname{PC}(G)$ of all PC-maps form a group under composition of functions.

Let $n\in\N\cup\{\infty\}$. Denote by $\UT(n,F)$ the group of upper unitriagular matrices over a field~$F$.
In case $n=\infty$ rows and columns of a matrix are indexed by natural numbers.
We will focus on PC-maps on the group $G=\UT(n,F)$. Such maps were classified
up to so called \emph{almost identity} PC-maps by M.\,Chen, D.\,Wang, and H.\,Zhai \cite{ChenWZ}
 in finite dimensional case and by R.\,Slowik in~\cite{Slowik}
for $n=\infty$. An almost identity map $\ph : \UT(n,F) \to \UT(n,F)$  is a map,
preserving elementary transvections
$t_{ij}(\alpha)$, i.\,e. for all  $i<j$, $\alpha \in F$ we have
\begin{equation}\label{ai}
\ph(t_{ij}(\alpha))= t_{ij}(\alpha)
\end{equation}

In the finite dimensional case examples of a PC-map which is not an automorphism were given in~\cite{ChenWZ}, remark after Lemma~2.1.
In infinite dimensional case no nontrivial almost identity PC-maps was known.
R.\,Slowik conjectured that
the map $\ph(a) = e - \sum_{i<j} (a^{-1})_{ij}e_{ij}$ was a nontrivial PC-map.
However, it is easy to compute that with this $\ph$ already $(1,3)$-entry of the left and the
right hand sides of (\ref{comm}) are not equal.

The current note is to show that an almost identity PC-map $\ph$ on the group $\UT(n,F)$
has the form $\ph(a)=af(a)$, where $f$ is a function from $\UT(n,F)$ to its center $C$.
In particular, since the center of $\UT(\infty,F)$ is trivial, any almost identity PC-map
of this group is identity. Combining this fact with the result of R.\,Slowik we see that all
PC-maps of the group $\UT(\infty,F)$ are automorphisms.
In any case the result of the current note completes the classification of PC-maps
of the group of upper unitriangular matrices.

\textbf{Notation.}
The following notation will be used thoughout the article.
The identity element of a matrix group is denoted by $e$.
More precisely,
$e=e_{\infty}$ denotes the infinite $\mathbb{N} \times \mathbb{N}$ identity matrix,
whereas $e=e_n$ is the $n \times n$ identity matrix. The matrix unit $e_{ij}$ is a matrix with $1$ in
position $(i,j)$ and $0$ elsewhere. An elementary transvection is a matrix
$t_{ij}(\alpha)= e + \alpha e_{ij}$,
where $\alpha \in F$. The set of all rows $(u_1, \ldots , u_n)$, where $u_i \in F$, is denoted by $^n\!F$
and the set of all columns $(u_1, \ldots , u_n)^t$ by $F^n$. If $a=(a_{ij})$ is an invertible matrix,
then the entries of its inverse are denoted by $a'_{ij}$ .
We use the following notation for rows and columns of a matrix $a$:
$a_{i*}$ and $a'_{i*}$ are the $i$th rows of $a$ and $a^{-1}$ respectively, whereas
$a_{*i}$ and $a'_{*i}$ are the $i$th columns of these matrices.
By convention we put $\infty+m=\infty$ for all $m\in\mathbb Z$.

By $\T(n,F)$ we denote the group of all upper triangular matrices
and by $\UT(n,F)$ its subgroup, consisting of matrices with 1 on the diagonal places.
If $n$ is finite, then $\T(n,F)$ is a standard Borel subgroup of $\GL(n,F)$ and
$\UT(n,F)$ is its unipotent radical. If $n=\infty$, then these groups are subgroups of the group of
the group of all invertible column finite matrices, which we denote by $\GL(\infty,F)$.
In both cases $\T(n,F)$ is the normalizer of $\UT(n,F)$ in $\GL(n,F)$.

The center of a group $G$ is denoted by $C(G)$. For the group of our main interest $G=\UT(n,F)$
we write $C$ instead of $C(G)$. Clearly, if $n=\infty$, then $C$ is trivial.

\section{Standard PC-maps}

Here we describe all known PC-maps on $\UT(n,F)$. Most of them are standard automorphisms that
were introduced already in classical description of automorphisms of finite dimensional general
linear group over fields (see \cite{SW,R,D}).
PC-maps that are not automorphisms were found in~\cite{ChenWZ} but our
terminology differs from the terminology of this article.

\textbf{Quasi-inner automorphims}.
Define a \textit{quasi-inner automorphism} as a conjugation by an element from $\T(n,F)$.
Since $\T(n,F)$ is a semidirect product of $\UT(n,F)$ with the group of all diagonal matrices,
a quasi-inner automorphism is a composition of an inner and a diagonal automorphisms in the terminology
of~\cite{ChenWZ} and~\cite{Slowik}. Clearly, the set of all quasi-inner automorphisms is a subgroup
isomorphic to the quotient of $\T(n,F)$ by the centralizer of $\UT(n,F)$.

\textbf{Field automorphisms}.
Since $\UT(n,_-)$ is a functor, a field automorphism $\theta:F\to F$ induces an automorphism
$\UT(n,F)\to\UT(n,F)$. By abuse of language this automorphisms will be called \textit{field
automorphisms} of $\UT(n,F)$. Clearly, the set of all field automorphisms of $\UT(n,F)$
is a group.

\textbf{Graph automorphism}.
If $n$ is finite, then the automorphism of the Dynkin diagram of the root system of
$\GL(n,F)$ induces an anti-automorphism of $\GL(n,F)$. Composing this map with taking
inverse we get the so-called graph automorphism of $\GL(n,F)$.
In other words, graph automorphism of $\GL(n,F)$ is the composition of transposition, taking inverse,
and conjugation by the matrix with 1 on the side diagonal and zeros elsewhere.
Clearly this map leaves $\UT(n,F)$ invariant. The induced automorphism of $\UT(n,F)$
will be also called the \textit{graph automorphism}.
The identity map is also a (trivial) graph automorphism. Since the nontrivial graph automorphism
is an involution, the subgroup of graph automorphisms consists of 2 elements.

\textbf{Central PC-maps}.
Let $C(G)$ be the center of a group $G$.
A map $\psi:G\to G$ is called a \textit{central map} if $\psi(a)=af(a)$ for some function $f:G\to C(G)$
and all $a\in G$. For $a,b\in G$ and a central map $\psi$ we have
$$
[\psi(a),\psi(b)]=[af(a),bf(b)]=[a,b].
$$
Therefore a central map is a PC-map iff it preserves all commutators.

Denote by $C$ the center of the group $\UT(n,F)$. If $n=\infty$, then $C$ is trivial, otherwise
$$
C=\{t_{1n}(\alpha)\mid\alpha\in F\}.
$$
In other words,
$$
a\equiv b\mod C \text{ iff } a_{ij}=b_{ij} \text{ for all } (i,j)\ne(1,n).
$$
Thus, if $\ph:\UT(n,F)\to\UT(n,F)$ is a central PC-map, then $\ph(a)=at_{1n}\bigl(f(a)\bigr)$,
where $f$ is a function $\UT(n,F)\to F$. It is well-known that the commutator subgroup of
$\UT(n,F)$ consists of all matrices $a\in\UT(n,F)$ such that $a_{i\,i+1}=0$ for all $i=1,\dots,n-1$.
It is easy to see that any such matrix is a single commutator.

\begin{lem}\label{commutators}
Let $a\in\UT(n,F)$ be such that $a_{i\,i+1}=0$ for all $i$.
Then there exist matrices $b,c\in\UT(n,F)$ such that $a=[b,c]$.
\end{lem}

\begin{proof}
This Lemma was proved in \cite{HolGupta} (Lemma 2.2) for $n=\infty$ and any associative ring $R$.
We note that the proof is valid also for $n \in \mathbb{N}$.
\end{proof}

In view of Lemma~\ref{commutators} a central map $\ph(a)=at_{1n}\bigl(f(a)\bigr)$, where
$f$ is a map $\UT(n,F)\to F$, is a PC-map iff $f(a)=0$ for all matrices $a$ with $a_{12}=\dots=a_{n-1\,n}=0$.
The definition of central PC-map in~\cite{ChenWZ} differs from our definition. The difference is
that in~\cite{ChenWZ} the map $f$ depends only on $a_{12},\dots,a_{n-1\,n}$.
It is easy to see that our central PC-map is a composition of a central PC-map of~\cite{ChenWZ}
and an almost identity map.
It is clear that the set of all central PC-maps is a subgroup of $\operatorname{PC}\bigl(\UT(n,F)\bigr)$.
Since a PC-map preserves the center of a group (see~\cite[Lemma~3.1(5)]{ChenWZ}),
this subgroup is normal.

\textbf{Subcentral PC-maps}.
Denote by $C_2$ the second center of $\UT(n,F)$, i.\,e. the preimage of the center of $\UT(n,F)/C$
under the reduction homomorphism. Recursively, define $C_m$ to be the preimage of the center
of the group $\UT(n,F)/C_{m-1}$ under the reduction homomorphism.
If $n=\infty$, then $C_2$ is trivial, otherwise
\begin{align*}
&C_2=\{t_{1\,n-1}(\alpha)t_{1n}(\beta)t_{2n}(\gamma)\mid\alpha,\beta,\gamma\in F\}\text{ and}\\
&a\equiv b\mod C_2 \text{ iff } a_{ij}=b_{ij} \text{ for all } (i,j)\notin\{(1,n-1),(1,n),(2,n)\}.
\end{align*}
Let $\psi:\UT(n,F)\to\UT(n,F)$ has the form $\psi(a)=af(a)$, where $f$ is a function from $\UT(n,F)$
to $C_2$. Then $\psi$ is called a \textit{subcentral map}.
It is easy to see that a subcentral PC-map preserves all double commutators $\bigl[a,[b,c]\bigr]$.
Similarly to Lemma~\ref{commutators} we can prove that a matrix $a\in\UT(n,F)$ is a double commutator
iff $a_{i\,i+1}=a_{i\,i+2}=0$ for all possible values of $i$. Hence such matrices are preserved by
a subcentral PC-map (cf. Lemma~\ref{1row}).

Clearly, a central map is a subcentral map. There is one more kind of subcentral maps that are
included into another set of standard PC-maps. Namely, a conjugation by an element from $C_3$ is a subcentral
map. To avoid this overlapping we define a \textit{standard subcentral PC-map} to be a map of the form
$$
\ph(a)=t_{2n}(\alpha a_{12})at_{1\,n-1}(\beta a_{n-1\,n}),\text{ where }\alpha,\beta\in F.
$$
A standard subcentral PC-map is exactly what was called a subcentral map in~\cite{ChenWZ}.
It is easy to see that the set of all subcentral PC-maps is a subgroup of
$\operatorname{PC}\bigl(\UT(n,F)\bigr)$. At the end of article we show that it is
normal. The set of all standard subcentral PC-maps
is a subgroup either, but it is not normal. To show this it suffices to consider
a conjugate to a standard subcentral PC-map by an appropriate inner automorphism.

\textbf{Permutable PC-maps}.
Since $\UT(3,F)$ coincides with its second center, any bijection $\UT(3,F)\to\UT(3,F)$
is subcentral. Therefore the case $n=3$ is an exception. In this case there
exists a more general kind of subcentral PC-maps than standard subcentral PC-maps.
Following~\cite{ChenWZ} we call them permutable. A \textit{permutable PC-map} is
a map of the form
$$
\ph(a)=
\begin{pmatrix}
1&\alpha a_{12}+\beta a_{23}& (\alpha\delta-\beta\gamma)a_{13}\\
0&              1           & \gamma a_{12}+\delta a_{23}     \\
0&              0           &              1                  \\
\end{pmatrix},
$$
where $\alpha,\beta,\gamma,\delta\in F$ and $\alpha\delta-\beta\gamma\ne0$.
\section{Not at the edge}\label{elemSec}

\textbf{In the rest of the article $\ph$ denotes an almost identity PC-map $\UT(n,F)\to \UT(n,F)$.}
In this section we show that all entries of $\ph(a)$ coincide with the corresponding entries
of $a$ except the first row and the last column.
During the proof we use subgroups $\UP_k$ of $\UT(n,F)$, consisting of matrices of the form
$\left(\begin{smallmatrix}e_k & \star\\ 0& e\end{smallmatrix}\right)$, where $\star$ denotes a
matrix with $k$ rows (possibly of infinite length). It is easy to see that $\UP_k$ is normal in
$\UT(n,F)$. In a finite dimensional case $\UP_k$ is the unipotent radical of $k$th standard
parabolic subgroup of $\operatorname{GL}(n,F)$.

\begin{lem}\label{elem}
Let $a\in\UT(n,F)$. Then $\ph(a)_{ij}=a_{ij}$ and $\ph(a)'_{ij}=a'_{ij}$ for all $j\le n-1$ and $i\ge2$.
\end{lem}

\begin{proof}
In finite-dimensional case the first equation was proved in~\cite[Lemma~2.1]{ChenWZ}.
Our proof here is quite similar.
Let $i,j$ be natural numberes, $i\ge2$ and $j\le n-1$.
Note that for an arbitrary matrix $a\in \UT(n,F)$ we have
$$
[t_{1i}(-1),a]=t_{1i}(-1)(e+a_{*1}a'_{i*})
$$
and since $a_{*1}=e_{*1}$, then
$$
\bigl[[t_{1i}(-1),a],t_{j\,j+1}(1)\bigr]=t_{1\,j+1}(a'_{ij}).
$$
The same equation holds for $\ph(a)$, hence
\begin{multline*}
t_{1\,j+1}(a'_{ij})=\ph\Bigl(t_{1\,j+1}(a'_{ij})\Bigr)=
\ph\Bigl(\bigl[[t_{1i}(-1),a],t_{j\,j+1}(1)\bigr]\Bigr)=\\
\bigl[[t_{1i}(-1),\ph(a)],t_{j\,j+1}(1)\bigr]=t_{1\,j+1}\bigl(\ph(a)'_{ij}\bigr).
\end{multline*}

It follows that $\ph(a)'_{ij}=a'_{ij}$ for all $i\ge2$ and $j\le n-1$.
If $n=\infty$, this implies inclusion $\ph(a)^{-1}a\in \UP_1$. Conjugate this equation by~$a$.
Since $\UP_1$ is normal in $\UT(n,F)$, we have $a\ph(a)^{-1}\in \UP_1$.
Thus $a$ can differ from $\ph(a)$ only in the first row.
If $n$ is finite, we repeat the above computation using the last row instead of the first column.
Specifically, for $a\in\UT(n,F)$ we have
$$
[t_{jn}(-1),a]=t_{jn}(-1)(e+a_{*j}a'_{n*})
$$
and since $a'_{n*}=e_{n*}$, then
$$
\bigl[t_{i-1\,i}(1),[t_{jn}(-1),a]\bigr]=t_{i-1\,n}(a_{ij}).
$$
The same equation holds for $\ph(a)$, hence
\begin{multline*}
t_{i-1\,n}(a_{ij})=\ph\Bigl(t_{i-1\,n}(a_{ij})\Bigr)=
\ph\Bigl(\bigl[t_{i-1\,i}(1),[t_{jn}(-1),a]\bigr]\Bigr)=\\
\bigl[t_{i-1\,i}(1),[t_{jn}(-1),\ph(a)]\bigr]=t_{i-1\,n}\bigl(\ph(a)_{ij}\bigr).
\end{multline*}

Thus, in any case $\ph(a)_{ij}=a_{ij}$ for all $i\ge2$ and $j\le n-1$. 
\end{proof}

\section{Reduction to subcentral maps}\label{InftySec}

In this section we prove that an almost identity PC-map $\ph$ is subcentral.
Clearly, this is already enough to settle the infinite dimensional case.
In view of Lemma~\ref{elem} it suffices to show that $\ph(a)_{1k}=a_{1k}$ and $\ph(a)_{ln}=a_{ln}$
for all $a\in \UT(n,F)$, $k\le n-2$ and $l\ge3$.
We prove the first equation by induction on $k$, the proof of the second one is similar
and will be left to the reader.
More precisely, for the following conditions we prove implications
$X_k\implies Y_k\implies Z_k\implies X_{k+1}$, where $k\le n-3$
(for the second and the third implications we assume additionally $Z_{k-1}$ and $X_k$
respectively).

\begin{itemize}
\item{$X_k$:}
$\ph(a)_{1i}=a_{1i}$ for all $i=2,\dots,k$ ($k\le n-2$).
\item{$Y_k$:}
Let $y=t_{k+1\,k+2}(\beta)\prod_{i=1}^{k-1}t_{ik}(\alpha_i)$
for some $\beta,\alpha_1,\dots,\alpha_{k-1}\in F$ ($k\le n-3$). Then
$\ph(y)\in yC$.
\item{$Z_k$:}
$\ph$ preserves matrices of the form $z=\prod_{i=1}^{k}t_{i\,k+3}(\gamma_i)$ for all
elements $\gamma_1,\dots,\gamma_{k}\in F$  ($k\le n-3$).
\end{itemize}

Note that for $k=1$ all the conditions are trivial: $X_1$ is empty, whereas $Y_1$ and $Z_1$ follows from the
fact that $\ph$ is almost identity.

\begin{lem}\label{X-Y}
$X_k\implies Y_k$ for all $k\le n-3$.
\end{lem}

\begin{proof}
Let $b=[t_{j\,j+1}(-1),y]$, where $j>k$.
Since for all $j>k>i$ transvections $t_{ik}(\alpha_i)$ and $t_{j\,j+1}(1)$ commute, we have
$b=t_{k+1\,k+3}(\beta)$ for $j=k+2$ and $b=e$ otherwise.
Applying $\ph$ to both sides of this formula we see that $\ph(b)=[t_{j\,j+1}(-1),\ph(y)]$ equals
$t_{k+1\,k+3}(\beta)$ if $j=k+2$ and $e$ otherwise. In both cases the first row of $\ph(b)$
coincides with the first row of the identity matrix.
On the other hand, calculation shows that $0=\ph(b)_{1\,j+1}=\ph(y)_{1j}$ for all $k<j\le n-1$.
Hence $\ph(y)_{1j}=y_{1j}$ for all $k+1<j\le n-1$ and
by condition $X_k$ we have $\ph(y)_{1j}=y_{1j}$ for all $j\le k$.
For $n=\infty$ these equations together with Lemma~\ref{elem} already show that $\ph(y)=y$.

Let $\UT_{n-1}$ be the subgroup of $\UT(n,F)$ consisting of all matrices whose last column coincides
with the last column of the identity matrix.
If $n$ is finite, consider the matrix $d=[y,t_{1j}(1)]$ for all $2\le j\le n-1$.
It is easy to see that $d$ is a transvection or the identity matrix and belongs to $\UT_{n-1}$.
(here we use that $k+2<n$). Therefore $\ph(d)=[\ph(y),t_{1j}(1)]\in\UT_{n-1}$.
Notice that $0=\ph(d)_{1n}=\ph(y)'_{jn}$ for all $2\le j\le n-1$.
In other words, $\ph(y)'$ belongs to $\UT_{n-1}C$, hence $\ph(y)\in\UT_{n-1}C$.
It follows that $\ph(y)_{jn}=y_{jn}$ for all $2\le j\le n-1$ and
$\ph(y)_{jl}=y_{jl}$ for all $j$ and $l\le n-1$ by Lemma~\ref{elem} and the first paragraph
of the proof. Thus $\ph(y)$ is congruent to $y$ modulo the center. 
\end{proof}

The prove of Lemma~\ref{elem} can be called ``extraction of transvections". The proof of the next statement
uses another idea, which we name ``construction of required elements''.
Namely, to prove that $\ph(z)=z$ we construct $z$ as a commutator of elements, preserved by $\ph$.

\begin{lem}\label{Y-Z}
$Z_{k-1}\& Y_k\implies Z_k$ for all $2\le k\le n-3$.
\end{lem}

\begin{proof}
In the proof we use matrices $y$ and $z$ defined in conditions $Y_k$ and $Z_k$.
If $\gamma_k=0$, then $Z_k$ follows from $Z_{k-1}$.
Indeded, by $Z_{k-1}$ the map $\ph$ preserves the matrix
$z'=\prod_{i=1}^{k-1}t_{i\,k+2}(\gamma_i)$, and $z=[z',t_{k+2\,k+3}(1)]$.
Since $\ph$ preserves the elementary transvections and commutators, we have $\ph(z)=z$.

Now, let $\gamma_k\ne0$. Calculation shows that
\begin{align*}
&[y,t_{k\,k+1}(-1)]=\prod_{i=1}^{k-1}\bigl(t_{i\,k+1}(-\alpha_i)t_{i\,k+2}(\alpha_i\beta)\bigr)
\cdot t_{k\,k+2}(\beta)\text{ and }\\
\big[&[y,t_{k\,k+1}(-1)],t_{k+2\,k+3}(1)\big]=\prod_{i=1}^{k-1}t_{i\,k+3}(\alpha_i\beta)
\cdot t_{k\,k+3}(\beta).
\end{align*}

By condition $Y_k$ we have $\ph(y)\in yC$ and the central factor vanishes after taking commutator
of $\ph(y)$ with an arbitrary element.
Hence $\ph$ preserves the left hand side of the latter equation.
It remains to notice that for $\beta=\gamma_k$ and $\alpha_i=\gamma_i/\gamma_k$
the right hand side of this equation coincides with $z$. 
\end{proof}

\begin{lem}\label{Z-X}
$X_{k}\& Z_k\implies X_{k+1}$ for all $k\le n-3$.
\end{lem}

\begin{proof}
Let $a\in \UT(n,F)$. Then
$$
[a,t_{k+1\,k+2}(1)]=(e+a_{*\,k+1}a'_{k+2\,*})t_{k+1\,k+2}(-1)\in
\prod_{i=1}^k t_{i\,k+2}(a_{i\,k+1})\UP_{k+2}.
$$
Since $\UP_{k+2}$ commutes with $t_{k+2\,k+3}(1)$, we have
$$
\bigl[[a,t_{k+1\,k+2}(1)],t_{k+2\,k+3}(1)\bigr]=\prod_{i=1}^k t_{i\,k+3}(a_{i\,k+1}).
$$
Recall that by condition $Z_k$ the map $\ph$ preserves the right hand side of the latter equation.
Therefore,
\begin{multline*}
\prod_{i=1}^k t_{i\,k+3}\bigl(\ph(a)_{i\,k+1})\bigr)=
\bigl[[\ph(a),t_{k+1\,k+2}(1)],t_{k+2\,k+3}(1)\bigr]=\\
\ph\Bigl(\bigl[[a,t_{k+1\,k+2}(1)],t_{k+2\,k+3}(1)\bigr]\Bigr)=
\prod_{i=1}^k t_{i\,k+3}(a_{i\,k+1}).
\end{multline*}
In particular $\ph(a)_{1\,k+1}=a_{1\,k+1}$. Together with condition $X_{k}$ this implies $X_{k+1}$.
\end{proof}

\begin{cor}\label{subcentral}
An almost identity PC-map $\ph:\UT(n,F)\to\UT(n,F)$ is a subcentral map.
\end{cor}

\begin{proof}
By Lemma~\ref{elem} $\ph(a)_{ij}=a_{ij}$ for all $i\ge2$ and $j\le n-1$.
Using Lemmas~\ref{X-Y}, \ref{Y-Z}, and~\ref{Z-X}, by induction on $k=1,\dots,n-2$ one shows
that $\ph(a)_{1k}=a_{1k}$ for all $k\le n-2$. If $n=\infty$ this means that $\ph(a)=a$.

If $n$ is finite, one proves similarly that
$\ph(a)_{ln}=a_{ln}$ for all $l\ge 3$. Thus, $\ph(a)$ differs from $a$ at most in three
positions: $(1,n-1)$, $(1,n)$ and $(2,n)$. As we have already noticed this is equivalent to
saying that $\ph(a)$ is congruent to $a$ modulo $C_2$. 
\end{proof}

\begin{thm}\label{InftyThm}
Every almost identity PC-map on $\UT(\infty,F)$ is identity.
\end{thm}

Combining the above result with the main theorem of~\cite{Slowik} we obtain the
following statement.

\begin{cor}\label{InftyCor}
Let $F\not\cong\mathbb F_2$ be a field.
The group $\operatorname{PC}\bigl(\UT(n,F)\bigr)$ coincides with the
automorphism group $\operatorname{Aut}\bigl(\UT(n,F)\bigr)$.
Every PC-map of $\UT(n,F)$ is a composition of a quasi-inner automorphism and a field automorphism.
\end{cor}

\section{In dimension 3}

From now on we assume that $n$ is finite.
Clearly, if $n\le2$, then the group $\UT(n,F)$ is abelian. Hence any map $\UT(n,F)\to\UT(n,F)$
preserving $e$ is a central PC-map. For the case $n=3$, since $\UT(3,F)=C_2$,
the results of sections~\ref{elemSec} and~\ref{InftySec} are empty.
Actually this case is much simpler than $n\ge4$ but it requires separate consideration.

\begin{lem}
An almost identity PC-map $\UT(3,F)\to\UT(3,F)$ is central.
\end{lem}

\begin{proof}
Let $\ph:\UT(3,F)\to\UT(3,F)$ be an almost identity PC-map given by the formula
$$
\ph(a)=
\begin{pmatrix} 1&\alpha'&\beta' \\
                0&   1   &\gamma'\\
                0&   0   &  1    \end{pmatrix},\text{ where }
a=\begin{pmatrix} 1&\alpha &\beta  \\
                  0&   1   &\gamma \\
                  0&   0   &  1    \end{pmatrix}.
$$
Since $[t_{12}(1),a]=t_{13}(\gamma)$, we have
$$
t_{13}(\gamma)=\ph\bigl(t_{13}(\gamma)\bigr)=\ph\bigl([t_{12}(1),a]\bigr)=
[t_{12}(1),\ph(a)]=t_{13}(\gamma').
$$
It follows that $\gamma'=\gamma$. Similarly, $\alpha'=\alpha$. Hence $\ph(a)$ is congruent to $a$ modulo
the center.
\end{proof}

Combining the above result with the assertion of~\cite[Theorem~2.2]{ChenWZ} for $n=3$ we obtain
a complete description of PC-maps on $\UT(3,F)$.

\begin{thm}
Let $F$ be a field of characteristic not $2$. Then a PC-map $\UT(3,F)\to\UT(3,F)$ is a composition
of a permutable PC-map, a field automorphism and a central PC-map.
\end{thm}

\section{Reduction to central maps}

From now on we assume that $4\le n<\infty$.
We prove that a subcentral almost identity PC-map must be central.
In this section we write matrices in a block form with respect to partition $n=1+(n-1)$.
First we study the image of a matrix
$\left(\begin{smallmatrix}1&u\\ 0&1\end{smallmatrix}\right)$ under an almost identity
PC-map $\ph$.

\begin{lem}\label{1row}
Let $u\in\,^{n-1}\!F$ and
$a=\left(\begin{smallmatrix}1&u\\ 0&e_{n-1}\end{smallmatrix}\right)$.
Then $\ph(a)=at_{1n}\bigl(f(u)\bigr)$ for some function $f:\,^{n-1}\!F\to F$.
If $u_1=u_2=0$, then $f(u)=0$.
\end{lem}

\begin{proof}
By Lemma~\ref{subcentral} $\ph$ is a subcentral map, i.\,e. it can differ from $a$ only in positions
$(1,n-1)$, $(2,n-1)$, and $(1,n)$. We have
$\ph\bigl([a,t_{n-1\,n}(1)]\bigr)=\ph\bigl(t_{1n}(a_{1\,n-1})]\bigr)=t_{1n}(a_{1\,n-1})$. On the other hand this
matrix is equal to $[\ph(a),t_{n-1\,n}(1)]=t_{1n}(\ph(a)_{1\,n-1})$. Therefore,
$\ph(a)_{1\,n-1}=a_{1\,n-1}$. Similarly, $\ph(a)_{2n}=a_{2n}$. This proves the first assertion of the lemma.

If $u_1=u_2=0$, then the matrix $a$ equals to the double commutator
$$a=\bigl[t_{12}(1),[t_{23}(1),\prod_{i=4}^nt_{3i}(u_{i-1})]].$$ Hence
$\ph(a)=a$, i.\,e. $f(u)=0$. 
\end{proof}

\begin{lem}
Suppose that $\ph$ is given by the formula
$$
\ph(b)=bt_{2n}\bigl(g(b)\bigr)t_{1\,n-1}\bigl(h(b)\bigr)t_{1n}\bigl(k(b)\bigr)
$$
For some functions $g,h,k:\UT(n,F)\to F$.
Then $g(b)=\alpha b_{23}$ and $h(b)=\beta b_{n-2\,n-1}$ for some $\alpha,\beta\in F$.
\end{lem}

\begin{proof}
Let $a$ and $u$ be as in the previous lemma. Then $\ph(a)=at_{1n}\bigl(f(u)\bigr)$.
Let $\tilde b\in\UT(n-1,F)$, $v\in\,^{n-1}\!F$, and
$b=\left(\begin{smallmatrix}1&v\\ 0&\tilde b\end{smallmatrix}\right)$.
Then we have
$$
[a,b]=\begin{pmatrix}1&u(e-{\tilde b}^{-1})\\ 0&e\end{pmatrix}\text{, hence }
\ph\bigl([a,b]\bigr)=
[a,b]\cdot t_{1n}\Bigl(f\bigl(u(e-{\tilde b}^{-1})\bigr)\Bigr).
$$
On the other hand this matrix is equal to
\begin{multline*}
\Bigl[\ph(a),\ph(b)\Bigr]=
\Bigl[at_{1n}\bigl(f(u)\bigr),bt_{2n}\bigl(g(b)\bigr)t_{1\,n-1}\bigl(h(b)\bigr)t_{1n}\bigl(k(b)\bigr)\Bigr]=\\
\Bigl[a,bt_{2n}\bigl(g(b)\bigr)\Bigr]=\Bigl[a,b\Bigr]\cdot\Bigl[a,t_{2n}\bigl(g(b)\bigr)\Bigr]=
\Bigl[a,b\Bigr]\cdot t_{1n}\bigl(u_1g(b)\bigr).
\end{multline*}

It follows that
\begin{equation}\label{EqForF}
f\bigl(u(e-{\tilde b}^{-1})\bigr)=u_1g(b)
\end{equation}
for all $b\in\UT(n,F)$ and $u\in\,^{n-1}\!F$.
Let $\hat c\in\UT(n-1,F)$ be the matrix with 1 on the main diagonal, $-1$ in all positions $(i,i+1)$, and zeros
elsewhere, so that $u(e-\hat c)=(0,u_1,\dots,u_{n-2})$.
Put $c=\left(\begin{smallmatrix}1&0\\ 0&\hat c^{-1}\end{smallmatrix}\right)$ and
denote $g(c)$ by $\alpha$.
From equation (\ref{EqForF}) with $c$ instead of $b$ we get $f(0,u_1,\dots,u_{n-2})=\alpha u_1$.

Note that $\tilde b'_{12}=b'_{23}=-b_{23}$, and therefore $u(e-{\tilde b}^{-1})=(0,u_1b_{23},\dots)$.
If $u_1=1$, then from (\ref{EqForF}) we get $g(b)=f(0,b_{23},\dots)=\alpha b_{23}$ as required.
The proof of the second equation is similar and will be left to the reader. 
\end{proof}

\begin{thm}\label{MAIN}
An almost identity PC-map $\ph:\UT(n,F)\to\UT(n,F)$ is a central map.
\end{thm}

\begin{proof}
Cases $n=\infty$ and $n=3$ have been already established in two previous sections.
If $4\le n<\infty$, by the previous lemma we have
\begin{equation}\label{SubcentralInner}
\ph(b)=bt_{2n}(\alpha b_{23})t_{1\,n-1}(\beta b_{n-2\,n-1})t_{1n}\bigl(k(b)\bigr)
\end{equation}
for some $\alpha,\beta\in F$ and a function $k:\UT(n,F)\to F$.
In particular,
$$
\ph\bigl(t_{23}(1)\bigr)=t_{23}(1)t_{2n}(\alpha)t_{1n}(\star).
$$
On the other hand $\ph\bigl(t_{23}(1)\bigr)=t_{23}(1)$ by the definition of an almost identity map.
Hence $\alpha=0$. Similarly one prove that $\beta=0$. Thus $\ph$ is a central map.
\end{proof}

Note that the map $\ph$ given by formula~\ref{SubcentralInner} is the composition of conjugation by
$t_{3n}(\alpha)t_{1\,n-2}(-\beta)$ with a central PC-map.

Now we formulate corollaries of~\cite[Theorem~2.2]{ChenWZ} and the previous theorem.
To use the result of~\cite{ChenWZ} we must assume that \textit{characteristic of $F$
is not equal to $2$}. We keep this assumption for the rest of the section.

\begin{thm}
Let $4\le n<\infty$ and $\operatorname{char}F\ne2$. Then a PC-map $\UT(n,F)\to\UT(n,F)$ is a composition of a
graph automorphism, a standard subcentral PC-map, a quasi-inner automorphism, a field automorphism,
and a central PC-map.
\end{thm}

Let $G=\UT(n,F)$. Denote by $\operatorname{SC-PC}(G)$ the set of all
subcentral PC-maps.

\begin{lem}
$\operatorname{SC-PC}(G)$ is a normal subgroup of $\operatorname{PC}(G)$.
\end{lem}

\begin{proof}
If $n\le3$, then $\operatorname{SC-PC}(G)=\operatorname{PC}(G)$.
If $n=\infty$, then the center is trivial and there are no nontrivial subcentral maps.
If $4\le n<\infty$, then by the previous theorem $\operatorname{PC}(G)$ is generated by
$\operatorname{SC-PC}(G)$ and $\operatorname{Aut}(G)$.
Therefore it suffices to show that $\operatorname{SC-PC}(G)$ is normalized by all automorphisms
of $G$. But this follows from the fact that the second center is a characteristic subgroup.
\end{proof}

The next assertion follows imediately from Corollary~\ref{InftyCor}, Theorem~\ref{MAIN},
and the previous lemma.

\begin{cor}
$\operatorname{PC}(G)=\operatorname{Aut}(G)\cdot\operatorname{SC-PC}(G)$.
\end{cor}


\end{document}